\documentclass[11pt,reqno]{amsart}
\usepackage{amsfonts}
\usepackage{amssymb,amsmath}
\usepackage{amscd}
\usepackage{hyperref}
\usepackage{graphicx}
\usepackage{epsfig}

\newtheorem{theorem}{Theorem}[section]

\newtheorem{example}{Example}[section]

\numberwithin{equation}{section}
\begin{document}

\title{On $\mathcal{C}$-Parallel Legendre Curves
in Non-Sasakian Contact Metric Manifolds}

\author[affil]{Cihan \"{O}zg\"{u}r}
\email{cozgur@balikesir.edu.tr}
\address{Bal\i kesir University, Department of Mathematics, 10145, Bal\i kesir, TURKEY}
\newcommand{\AuthorNames}{C. \"{O}ZG\"{U}R}

\newcommand{\FilMSC}{Primary 53C25; Secondary 53C40, 53A05}
\newcommand{\FilKeywords}{Contact metric manifold, Legendre
curve, $\mathcal{C}$-parallel mean curvature vector field, $\mathcal{C}$%
-proper mean curvature vector field}


\begin{abstract}
In $(2n+1)$-dimensional non-Sasakian contact metric
manifolds, we consider Legendre curves whose mean curvature vector fields
are $\mathcal{C}$-parallel or $\mathcal{C}$-proper in the tangent or normal
bundles. We obtain the curvature characterizations of these curves. Moreover,
we give some examples of these kinds of curves which satisfy the conditions
of our results.
\end{abstract}

\maketitle

\makeatletter
\renewcommand\@makefnmark%
{\mbox{\textsuperscript{\normalfont\@thefnmark)}}}
\makeatother

\section{Introduction} \label{sect-introduction}

In \cite{Chen-89} and \cite{Chen-95}, Chen studied submanifolds whose mean
curvature vector fields $H$ satisfy the condition $\Delta H=\lambda H,$
where $\lambda $ is a non-zero differentiable function on the submanifold
and $\Delta $ denotes the Laplacian. Later, in \cite{Arroyo}, Arroyo, Barros
and Garay defined the notion of a submanifold with a proper mean curvature
vector field $H$ in the normal bundle as a submanifold whose mean curvature
vector field $H$ satisfies the condition $\Delta ^{\perp }H=\lambda H$,
where $\Delta ^{\perp }$ denotes the Laplacian in the normal bundle.
Furthermore, when the mean curvature vector field $H$ of the submanifold
satisfies the condition $\Delta H=\lambda H,$ they called the submanifold as
a \textit{submanifold with a proper mean curvature vector field}. In a
Riemannian space form, curves with a proper mean curvature vector field in
the tangent and normal bundles were studied in \cite{Arroyo}. In \cite{KA},
K\i l\i \c{c} and Arslan studied Euclidean submanifolds satisfying $\Delta
^{\perp }H=\lambda H.$  In \cite{KH}, \ Kocayi\u{g}it and Hac\i saliho\u{g}lu
studied curves satisying $\Delta H=\lambda H$ in a 3-dimensional Riemannian
manifold. For Legendre curves in Sasakian manifolds, same problems were
studied by Inoguchi in \cite{Inoguchi-2004}. In \cite{BB-92}, Baikoussis and
Blair considered submanifolds in Sasakian space forms $M(c)=(M,\varphi ,\xi
,\eta ,g)$. They defined the mean curvature vector field $H$ as \textit{$%
\mathcal{C}$-parallel }if $\nabla H=\lambda \xi ,$ where $\lambda $ is a
non-zero differentiable function on $M$ and $\nabla $ the induced
Levi-Civita connection. Later, in \cite{LSL}, Lee, Suh and Lee studied
curves with $\mathcal{C}$-parallel and $\mathcal{C}$-proper mean curvature
vector fields in the tangent and normal bundles. A curve $\gamma $ has \textit{$%
\mathcal{C}$-parallel mean curvature vector field }$H$\textit{\ }if $\nabla
_{T}H=\lambda \xi $, $\mathcal{C}$\textit{-proper mean curvature vector
field }$H$ if $\Delta H=\lambda \xi $, $\mathcal{C}$\textit{-parallel mean
curvature vector field }$H$\textit{\ in the normal bundle }if $\nabla
_{T}^{\perp }H=\lambda \xi $, $\mathcal{C}$\textit{-proper mean curvature
vector field }$H$ \textit{in the normal bundle} if $\Delta ^{\perp
}H=\lambda \xi $, where $\lambda $ is a non-zero differentiable function
along the curve $\gamma $, $T$ the unit tangent vector field of $\gamma $, $%
\nabla $ the Levi-Civita connection, $\nabla ^{\perp }$ the normal
connection \cite{LSL}.

Let $M=(M,\varphi ,\xi ,\eta ,g)$ be a contact metric manifold and $\gamma
:I\rightarrow M$ a Frenet curve in $M$ parametrized by the arc-length
parameter $s$. The contact angle $\alpha (s)$ is a function defined by $%
cos[\alpha (s)]=g(T(s),\xi )$. If $\alpha (s)$ is a
constant, then the curve is called a \textit{slant curve} \cite{CIL}. If $%
\alpha (s)=$ $\frac{\pi }{2},$ then $\gamma $ is called a \textit{%
Legendre curve }\cite{Blair}.

In \cite{LSL}, Lee, Suh and Lee studied slant curves with $\mathcal{C}$%
-parallel and $\mathcal{C}$-proper mean curvature vector fields in Sasakian
3-manifolds. In \cite{GO}, G\"{u}ven\c{c} and the present author studied $%
\mathcal{C}$-parallel and $\mathcal{C}$-proper slant curves in $(2n+1)$%
-dimensional trans-Sasakian manifolds. Since the paper \cite{GO} includes
the Legendre curves in Sasakian manifolds, in the present paper, we consider $%
\mathcal{C}$-parallel and $\mathcal{C}$-proper Legendre curves in $(2n+1)$%
-dimensional non-Sasakian contact metric manifolds.

The paper is organized as follows: In Section \ref{sect-cparallel} and
Section \ref{sect-cproper}, in non-Sasakian contact metric manifolds, we
consider Legendre curves with $\mathcal{C}$-parallel and $\mathcal{C}$%
-proper mean curvature vector fields, respectively. In the final section, we
give some examples of Legendre curves which support our theorems.

\section{Legendre Curves with $\mathcal{C}$-parallel Mean
Curvature Vector Fields} \label{sect-cparallel}

Let $M=(M,\varphi ,\xi ,\eta ,g)$ be a contact metric manifold. The contact metric structure of $M$ is said to be \textit{normal} if%
\begin{equation*}
\lbrack \varphi ,\varphi ](X,Y)=-2d\eta (X,Y)\xi ,
\end{equation*}%
where $[\varphi ,\varphi ]$ denotes the Nijenhuis torsion of $\varphi $ and $%
X,Y$ are vector fields on $M$.
 A normal contact metric manifold is called a
\textit{Sasakian manifold }\cite{Blair}.

Given a contact Riemannian manifold $M$, the operator $h$ is defined by $h=%
\frac{1}{2}\left( L_{\xi }\varphi \right) $, where $L$ denotes the Lie
differentiation. The operator $h$ is self adjoint and satisfies
\begin{equation*}
\begin{array}{ccc}
h\xi =0 & \text{and} & h\varphi =-\varphi h,%
\end{array}%
\end{equation*}%
\begin{equation}
\nabla _{X}\xi =-\varphi X-\varphi hX.  \label{nablaXksi}
\end{equation}%
In a Sasakian manifold, it is clear that
\begin{equation*}
\nabla _{X}\xi =-\varphi X.
\end{equation*}%
For more details about contact metric manifolds and their submanifolds, we
refer to \cite{Blair} and \cite{Yano}.

Let $\left( M,g\right) $ be an $n$-dimensional Riemannian manifold. A
unit-speed curve $\gamma :I\rightarrow M$ is said to be \textit{a Frenet
curve of osculating order }$r$, if there exists positive functions $%
k_{1},...,k_{r-1}$ on $I$ satisfying
\begin{eqnarray*}
T &=&\upsilon _{1}=\gamma ^{\prime }, \\
\nabla _{T}T &=&k_{1}\upsilon _{2}, \\
\nabla _{T}\upsilon _{2} &=&-k_{1}T+k_{2}\upsilon _{3}, \\
&&... \\
\nabla _{T}\upsilon _{r} &=&-k_{r-1}\upsilon _{r-1},
\end{eqnarray*}%
where $1\leq r\leq n$ and $T,\upsilon _{2},...,\upsilon _{r}$ are a $g$%
-orthonormal vector fields along the curve. The positive functions $%
k_{1},...,k_{r-1}$ are called\textit{\ curvature functions }and $\left\{
T,\upsilon _{2},...,\upsilon _{r}\right\} $ is called the \textit{Frenet
frame field}. A \textit{geodesic} is a Frenet curve of osculating order $%
r=1. $ A \textit{circle} is a Frenet curve of osculating order $r=2$ with a
constant curvature function $k_{1}$. A \textit{helix of order} $r$ is a
Frenet curve of osculating order $r$ with constant curvature functions $%
k_{1},...,k_{r-1}$. A helix of order $3$ is simply called a \textit{helix}.

Now let $(M,g)$ be a Riemannian manifold and $\gamma :I\rightarrow M$ a unit
speed Frenet curve of osculating order $r.$ By a simple calculations, it can
be easily seen that%
\begin{equation*}
\nabla _{T}\nabla _{T}T=-k_{1}^{2}\upsilon _{1}+k_{1}^{\prime }\upsilon
_{2}+k_{1}k_{2}\upsilon _{3},
\end{equation*}%
\begin{eqnarray*}
\nabla _{T}\nabla _{T}\nabla _{T}T &=&-3k_{1}k_{1}^{\prime }T+\left(
k_{1}^{\prime \prime }-k_{1}^{3}-k_{1}k_{2}^{2}\right) \upsilon _{2} \\
&&+\left( 2k_{1}^{\prime }k_{2}+k_{1}k_{2}^{\prime }\right) \upsilon
_{3}+k_{1}k_{2}k_{3}\upsilon _{4},
\end{eqnarray*}%
\begin{equation*}
\nabla _{T}^{\perp }\nabla _{T}^{\perp }T=k_{1}^{\prime }\upsilon
_{2}+k_{1}k_{2}\upsilon _{3},
\end{equation*}%
\begin{equation*}
\nabla _{T}^{\perp }\nabla _{T}^{\perp }\nabla _{T}^{\perp }T=\left(
k_{1}^{\prime \prime }-k_{1}k_{2}^{2}\right) \upsilon _{2}+\left(
2k_{1}^{\prime }k_{2}+k_{1}k_{2}^{\prime }\right) \upsilon
_{3}+k_{1}k_{2}k_{3}\upsilon _{4},
\end{equation*}%
(see \cite{GO}). Then we have
\begin{equation}
\nabla _{T}H=-k_{1}^{2}T+k_{1}^{\prime }\upsilon _{2}+k_{1}k_{2}\upsilon
_{3},  \label{nablaH}
\end{equation}%
\begin{eqnarray}
\Delta H &=&-\nabla _{T}\nabla _{T}\nabla _{T}T  \notag \\
&=&3k_{1}k_{1}^{\prime }T+\left( k_{1}^{3}+k_{1}k_{2}^{2}-k_{1}^{\prime
\prime }\right) \upsilon _{2}  \notag \\
&&-(2k_{1}^{\prime }k_{2}+k_{1}k_{2}^{\prime })\upsilon
_{3}-k_{1}k_{2}k_{3}\upsilon _{4},  \label{deltaH}
\end{eqnarray}%
\begin{equation}
\nabla _{T}^{\perp }H=k_{1}^{\prime }\upsilon _{2}+k_{1}k_{2}\upsilon _{3},
\label{nablaTH}
\end{equation}%
\begin{eqnarray}
\Delta ^{\perp }H &=&-\nabla _{T}^{\perp }\nabla _{T}^{\perp }\nabla
_{T}^{\perp }T  \notag \\
&=&\left( k_{1}k_{2}^{2}-k_{1}^{\prime \prime }\right) \upsilon _{2}-\left(
2k_{1}^{\prime }k_{2}+k_{1}k_{2}^{\prime }\right) \upsilon _{3}  \notag \\
&&-k_{1}k_{2}k_{3}\upsilon _{4},  \label{deltaTH}
\end{eqnarray}%
(see \cite{Arroyo}).

\bigskip Let $\gamma :I\subseteq
\mathbb{R}
\rightarrow M$ be a non-geodesic Frenet curve in a contact metric manifold $%
M $. From \cite{GO}, we give the following relations:

$i)$ $\gamma $ is a curve with $\mathcal{C}$-parallel mean curvature vector
field $H$ if and only if%
\begin{equation}
-k_{1}^{2}T+k_{1}^{\prime }\upsilon _{2}+k_{1}k_{2}\upsilon _{3}=\lambda \xi
;\text{ or}  \label{cparallel}
\end{equation}

$ii)$ $\gamma $ is a curve with $\mathcal{C}$-proper mean curvature vector
field $H$ if and only if%
\begin{equation}
3k_{1}k_{1}^{\prime }T+\left( k_{1}^{3}+k_{1}k_{2}^{2}-k_{1}^{\prime \prime
}\right) \upsilon _{2}-(2k_{1}^{\prime }k_{2}+k_{1}k_{2}^{\prime })\upsilon
_{3}-k_{1}k_{2}k_{3}\upsilon _{4}=\lambda \xi \text{; or}  \label{cproper}
\end{equation}

$iii)$ $\gamma $ is a curve with $\mathcal{C}$-parallel mean curvature
vector field $H$ in the normal bundle if and only if%
\begin{equation}
k_{1}^{\prime }\upsilon _{2}+k_{1}k_{2}\upsilon _{3}=\lambda \xi \text{; or}
\label{cparallelNORMAL}
\end{equation}

$iv)$ $\gamma $ is a curve with $\mathcal{C}$-proper mean curvature vector
field $H$ in the normal bundle if and only if%
\begin{equation}
\left( k_{1}k_{2}^{2}-k_{1}^{\prime \prime }\right) \upsilon _{2}-\left(
2k_{1}^{\prime }k_{2}+k_{1}k_{2}^{\prime }\right) \upsilon
_{3}-k_{1}k_{2}k_{3}\upsilon _{4}=\lambda \xi ,  \label{cproperNORMAL}
\end{equation}
where $\lambda $ is a non-zero differentiable function along the curve $%
\gamma .$

Now, let $\gamma :I\subseteq
\mathbb{R}
\rightarrow M$ be a non-geodesic Legendre curve of osculating order $r$ in an $n$%
-dimensional contact metric manifold. By the use of the definition of a
Legendre curve and (\ref{nablaXksi}), we have
\begin{equation}
\eta (T)=0,  \label{a}
\end{equation}%
\begin{equation}
\nabla _{T}\xi =-\varphi T-\varphi hT.  \label{c}
\end{equation}%
Differentiating (\ref{a}) and using (\ref{c}), we obtain%
\begin{equation}
k_{1}\eta (\upsilon _{2})=g(T,\varphi hT).  \label{etaE2}
\end{equation}

If the osculating order $r=2,$ then we have the following results:

\begin{theorem}
\label{propcaseIcparallelTANGENT}There does not exist a non-geodesic
Legendre curve $\gamma :I\subseteq
\mathbb{R}
\rightarrow M$ of osculating order $2$, which has $\mathcal{C}$-parallel
mean curvature vector field in a contact metric manifold $M$.
\end{theorem}

\begin{proof}
Assume that $\gamma $ have $\mathcal{C}$-parallel mean curvature vector
field. From (\ref{cparallel}), we have%
\begin{equation}
-k_{1}^{2}T+k_{1}^{\prime }\upsilon _{2}=\lambda \xi .  \label{caseIeq}
\end{equation}%
Then taking the inner product of (\ref{caseIeq}) with $T$, we find $k_{1}=0$%
, this means that $\gamma $ is a geodesic. This completes the proof.
\end{proof}

In the normal bundle, we can state the following theorem:

\begin{theorem}
\label{case1normalbundle}Let $\gamma :I\subseteq
\mathbb{R}
\rightarrow M$ be a non-geodesic Legendre curve of osculating order $2$ in a
non-Sasakian contact metric manifold. Then $\gamma $ has $\mathcal{C}$%
-parallel mean curvature vector field in the normal bundle if and only if
\begin{equation}
k_{1}=\pm g(\varphi hT,T),\text{ }\xi =\pm \upsilon _{2},\text{ }\lambda
=k_{1}^{\prime }.  \label{b4}
\end{equation}
\end{theorem}

\begin{proof}
Let $\gamma $ have $\mathcal{C}$-parallel mean curvature vector field in the
normal bundle. From (\ref{cparallelNORMAL}) we have%
\begin{equation}
k_{1}^{\prime }\upsilon _{2}=\lambda \xi .  \label{b1}
\end{equation}%
So we have%
\begin{equation*}
\lambda =\pm k_{1}^{\prime },
\end{equation*}%
\begin{equation}
\xi =\pm \upsilon _{2}.  \label{b2}
\end{equation}%
Differentiating (\ref{b2}), we find%
\begin{equation}
-\varphi T-\varphi hT=\mp k_{1}T,  \label{b3}
\end{equation}%
which gives us
\begin{equation*}
k_{1}=\pm g(\varphi hT,T).
\end{equation*}%
The converse statement is trivial. Then we complete the proof.
\end{proof}

If\textbf{\ }the osculating order $r\geq 3$, then similar to the proof of
Theorem \ref{propcaseIcparallelTANGENT}, we have the following theorem:

\begin{theorem}
\label{propcase2tangent}There does not exist a non-geodesic Legendre curve $%
\gamma :I\subseteq
\mathbb{R}
\rightarrow M$ of osculating order $r\geq 3$, which has $\mathcal{C}$%
-parallel mean curvature vector field in a contact metric manifold $M$.
\end{theorem}

In the normal bundle, we have the following theorem:

\begin{theorem}
\label{propcase2normal}Let $\gamma :I\subseteq
\mathbb{R}
\rightarrow M$ be a non-geodesic Legendre curve of osculating order $3$ in a
non-Sasakian contact metric manifold. Then $\gamma $ has $\mathcal{C}$%
-parallel mean curvature vector field in the normal bundle if and only if

\begin{equation*}
k_{1}\neq constant,
\end{equation*}%
\begin{equation*}
k_{2}=\mp \frac{k_{1}^{\prime }\sqrt{k_{1}^{2}-g(T,\varphi hT)^{2}}}{%
k_{1}g(T,\varphi hT)},
\end{equation*}%
\begin{equation*}
\xi =\frac{g(T,\varphi hT)}{k_{1}}\upsilon _{2}+\frac{k_{2}}{k_{1}^{\prime }}%
g(T,\varphi hT)\upsilon _{3}
\end{equation*}%
and%
\begin{equation*}
\lambda =\frac{k_{1}^{\prime }k_{1}}{g(T,\varphi hT)}
\end{equation*}%
or
\begin{equation*}
k_{1}=constant,
\end{equation*}%
\begin{equation*}
k_{2}=\sqrt{1+2g(T,hT)+g(hT,hT)}
\end{equation*}%
\begin{equation*}
\begin{array}{ccc}
\lambda =k_{1}k_{2} & \text{and} & \xi =\upsilon _{3}.%
\end{array}%
\end{equation*}
\end{theorem}

\begin{proof}
If $k_{1}\neq constant,$ then from (\ref{cparallelNORMAL}), we have%
\begin{equation}
k_{1}^{\prime }\upsilon _{2}+k_{1}k_{2}\upsilon _{3}=\lambda \xi \text{.}
\label{f1}
\end{equation}%
Then taking the inner product of (\ref{f1}) with $\upsilon _{2}$ and using (%
\ref{etaE2}), we find
\begin{equation}
k_{1}^{\prime }=\lambda \eta (\upsilon _{2})=\lambda \frac{g(T,\varphi hT)}{%
k_{1}}.  \label{f2}
\end{equation}%
This gives us
\begin{equation*}
\lambda =\frac{k_{1}^{\prime }k_{1}}{g(T,\varphi hT)}.
\end{equation*}%
Taking the inner product of (\ref{f1}) with $\upsilon _{3}$, we have%
\begin{equation}
\eta (\upsilon _{3})=\frac{k_{2}g(T,\varphi hT)}{k_{1}^{\prime }}.
\label{f3}
\end{equation}%
Since $\xi \in span\left\{ \upsilon _{2},\upsilon _{3}\right\} $, using (\ref%
{etaE2}) and (\ref{f3}), we get
\begin{equation*}
\xi =\frac{g(T,\varphi hT)}{k_{1}}\upsilon _{2}+\frac{k_{2}}{k_{1}^{\prime }}%
g(T,\varphi hT)\upsilon _{3}
\end{equation*}%
Since $\xi $ is a unit vector field, we obtain%
\begin{equation*}
k_{2}=\mp \frac{k_{1}^{\prime }\sqrt{k_{1}^{2}-g(T,\varphi hT)^{2}}}{%
k_{1}g(T,\varphi hT)}.
\end{equation*}
If $k_{1}=constant,$ then from (\ref{cparallelNORMAL}), we have%
\begin{equation*}
k_{1}k_{2}\upsilon _{3}=\lambda \xi ,
\end{equation*}%
which gives us $\lambda =k_{1}k_{2}$ and $\xi =\upsilon _{3}.$ So by a
differentiation of $\xi =\upsilon _{3},$ using (\ref{nablaXksi}), we have $%
-k_{2}\upsilon _{2}=-\varphi T-\varphi hT.$ Hence, we obtain
\begin{equation*}
k_{2}=\sqrt{1+2g(T,hT)+g(hT,hT)}.
\end{equation*}

The converse statement is trivial. This completes the proof of the theorem.
\end{proof}

\section{Legendre Curves with $\mathcal{C}$-proper Mean
Curvature Vector Fields} \label{sect-cproper}

If the osculating order $r=2,$ then we have the following theorems:

\begin{theorem}
\label{case1propTANGENT}Let $\gamma :I\subseteq
\mathbb{R}
\rightarrow M$ be a non-geodesic Legendre curve of osculating order $2$ in a
non-Sasakian contact metric manifold. Then $\gamma $ has $\mathcal{C}$%
-proper mean curvature vector field if and only if
\begin{equation*}
k_{1}=\pm g(T,\varphi hT)=\text{constant},
\end{equation*}%
\begin{equation*}
\xi =\pm \upsilon _{2}
\end{equation*}%
and%
\begin{equation*}
\lambda =g(T,\varphi hT)^{3}.
\end{equation*}
\end{theorem}

\begin{proof}
Let $\gamma $ have $\mathcal{C}$-proper mean curvature vector field. From (%
\ref{cproper}), we have%
\begin{equation}
3k_{1}k_{1}^{\prime }T+\left( k_{1}^{3}-k_{1}^{\prime \prime }\right)
\upsilon _{2}=\lambda \xi .  \label{j1}
\end{equation}%
Then taking the inner product of (\ref{j1}) with $T$, we have $%
k_{1}k_{1}^{\prime }=0$. Since $\gamma $ is not a geodesic, we obtain $%
k_{1}^{\prime }=0,$ which means that $k_{1}$ is a constant. Taking the inner
product of (\ref{j1}) with $\upsilon _{2}$, we have
\begin{equation*}
k_{1}^{3}-k_{1}^{\prime \prime }=\lambda \eta (\upsilon _{2}).
\end{equation*}%
Since $k_{1}$ is a constant, using (\ref{etaE2}), we get
\begin{equation}
\lambda =\frac{k_{1}^{4}}{g(T,\varphi hT)}.  \label{lamda1}
\end{equation}%
Furthermore, taking the inner product of (\ref{j1}) with $\xi $ and using (%
\ref{etaE2}), we have
\begin{equation}
\lambda =k_{1}^{2}g(T,\varphi hT).  \label{lamda2}
\end{equation}%
Then comparing (\ref{lamda1}) and (\ref{lamda2}), we obtain
\begin{equation*}
k_{1}=\mp g(T,\varphi hT).
\end{equation*}
Since $\xi \in span\left\{ \upsilon _{2}\right\} $ we have
\begin{equation}
\xi =\pm \upsilon _{2}.  \label{j2}
\end{equation}%
The converse statement is trivial. Hence, the proof is finished.
\end{proof}

In the normal bundle, we can state the following theorem:

\begin{theorem}
\label{case1propNORMALCPROPER}Let $\gamma :I\subseteq
\mathbb{R}
\rightarrow M$ be a non-geodesic Legendre curve of osculating order $2$ in a
non-Sasakian contact metric manifold. Then $\gamma $ is a curve with $%
\mathcal{C}$-proper mean curvature vector field in the normal bundle if and
only if

i) $k_{1}(s)=as+b,$ where $a$ and $b$ are arbitrary real constants and $\lambda
=0$ or

ii) $k_{1}=\mp g(T,\varphi hT),$ $\xi =\pm \upsilon _{2}$ and $\lambda
=k_{1}^{\prime \prime }.$
\end{theorem}

\begin{proof}
Let $\gamma $ have $\mathcal{C}$-proper mean curvature vector field in the
normal bundle. From (\ref{cproperNORMAL}) we have%
\begin{equation}
-k_{1}^{\prime \prime }\upsilon _{2}=\lambda \xi .  \label{h1}
\end{equation}%
Taking the inner product of (\ref{h1}) with $\upsilon _{2}$ and using (\ref%
{etaE2}), we have
\begin{equation}
\lambda =-\frac{k_{1}^{\prime \prime }k_{1}}{g(T,\varphi hT)}  \label{lamda3}
\end{equation}%
Taking the inner product of (\ref{h1}) with $\xi $ and using (\ref{etaE2}),
we find%
\begin{equation}
\lambda =-\frac{k_{1}^{\prime \prime }g(T,\varphi hT)}{k_{1}}  \label{lamda4}
\end{equation}%
Then comparing (\ref{lamda3}) and (\ref{lamda4}), we obtain either $%
k_{1}^{\prime \prime }=0$, in this case $k_{1}(s)=as+b,$ where $a$ and $b$
are arbitrary real constants and $\lambda =0$ or $k_{1}=\mp g(T,\varphi hT)$. If $%
k_{1}=\mp g(T,\varphi hT)$, it is easy to see that $\xi =\pm \upsilon _{2}$
and $\lambda =k_{1}^{\prime \prime }.$

The converse statement is trivial. This completes the proof of the theorem.
\end{proof}

If the osculating order $r=3,$ then we have the following theorems:

\begin{theorem}
\label{case2propTANGENT}Let $\gamma :I\subseteq
\mathbb{R}
\rightarrow M$ be a non-geodesic Legendre curve of osculating order $3$ in a
non-Sasakian contact metric manifold. Then $\gamma $ is a curve with $%
\mathcal{C}$-proper mean curvature vector field if and only if
\begin{equation*}
k_{1}=\text{ constant,}
\end{equation*}%
\begin{equation*}
\lambda =\frac{k_{1}^{2}\left( k_{1}^{2}+k_{2}^{2}\right) }{g(T,\varphi hT)},
\end{equation*}%
\begin{equation*}
\xi =\frac{g(T,\varphi hT)}{k_{1}}\upsilon _{2}-\frac{k_{1}k_{2}^{\prime }}{%
\lambda }\upsilon _{3}
\end{equation*}%
and%
\begin{equation*}
\eta (\upsilon _{2})^{2}+\eta (\upsilon _{3})^{2}=1.
\end{equation*}
\end{theorem}

\begin{proof}
Let $\gamma $ have $\mathcal{C}$-proper mean curvature vector field. Then,
from (\ref{cproper}), we have%
\begin{equation}
3k_{1}k_{1}^{\prime }T+\left( k_{1}^{3}+k_{1}k_{2}^{2}-k_{1}^{\prime \prime
}\right) \upsilon _{2}-(2k_{1}^{\prime }k_{2}+k_{1}k_{2}^{\prime })\upsilon
_{3}=\lambda \xi .  \label{k1}
\end{equation}%
Taking the inner product of (\ref{k1}) with $T$, we have $k_{1}k_{1}^{\prime
}=0$. Since $\gamma $ is not a geodesic, we find $k_{1}^{\prime }=0$, which
gives us $k_{1}$ is a constant. Now taking the inner product of (\ref{k1})
with $\upsilon _{2}$ and using (\ref{etaE2}), we have
\begin{equation*}
\lambda =\frac{k_{1}^{2}\left( k_{1}^{2}+k_{2}^{2}\right) }{g(T,\varphi hT)}.
\end{equation*}%
Taking the inner product of (\ref{k1}) with $\upsilon _{3}$, we have%
\begin{equation}
\eta (\upsilon _{3})=-\frac{k_{1}k_{2}^{\prime }}{\lambda }.  \label{eta(E3)}
\end{equation}%
Since $\xi \in span\left\{ \upsilon _{2},\upsilon _{3}\right\} ,$ using (\ref%
{etaE2}) and (\ref{eta(E3)}) we obtain
\begin{equation*}
\xi =\frac{g(T,\varphi hT)}{k_{1}}\upsilon _{2}-\frac{k_{1}k_{2}^{\prime }}{%
\lambda }\upsilon _{3}.
\end{equation*}%
Since $\xi $ is a unit vector field, we have $\eta (\upsilon _{2})^{2}+\eta
(\upsilon _{3})^{2}=1.$ The converse statement is trivial. So we get the result
as required.
\end{proof}

In the normal bundle, we can give the following result:

\begin{theorem}
\label{case2cproperNORMAL}Let $\gamma :I\subseteq
\mathbb{R}
\rightarrow M$ be a non-geodesic Legendre curve of osculating order $3$ in a
non-Sasakian contact metric manifold. Then $\gamma $ is a curve with $%
\mathcal{C}$-proper mean curvature vector field in the normal bundle if and
only if
\begin{equation*}
\lambda =\frac{k_{1}^{2}k_{2}^{2}-k_{1}k_{1}^{\prime \prime }}{g(T,\varphi
hT)},
\end{equation*}%
\begin{equation*}
\xi =\frac{g(T,\varphi hT)}{k_{1}}\upsilon _{2}-\frac{\left( 2k_{1}^{\prime
}k_{2}+k_{1}k_{2}^{\prime }\right) }{\lambda }\upsilon _{3}
\end{equation*}
and%
\begin{equation*}
\eta (\upsilon _{2})^{2}+\eta (\upsilon _{3})^{2}=1.
\end{equation*}
\end{theorem}

\begin{proof}
Let $\gamma $ have $\mathcal{C}$-proper mean curvature vector field in the
normal bundle. From (\ref{cproperNORMAL}), $\gamma $ is a Legendre curve with%
\begin{equation}
\left( k_{1}k_{2}^{2}-k_{1}^{\prime \prime }\right) \upsilon _{2}-\left(
2k_{1}^{\prime }k_{2}+k_{1}k_{2}^{\prime }\right) \upsilon _{3}=\lambda \xi .
\label{i1}
\end{equation}%
Taking the inner product of (\ref{i1}) with $\upsilon _{2}$ and using (\ref%
{etaE2}), we have%
\begin{equation*}
\lambda =\frac{k_{1}^{2}k_{2}^{2}-k_{1}k_{1}^{\prime \prime }}{g(T,\varphi
hT)}.
\end{equation*}%
Taking the inner product of (\ref{i1}) with $\upsilon _{3}$, we get%
\begin{equation}
\eta (\upsilon _{3})=-\frac{2k_{1}^{\prime }k_{2}+k_{1}k_{2}^{\prime }}{%
\lambda }.  \label{etaE32}
\end{equation}%
Since $\xi \in span\left\{ \upsilon _{2},\upsilon _{3}\right\},$ using (\ref%
{etaE2}) and (\ref{etaE32}), we obtain
\begin{equation*}
\xi =\frac{g(T,\varphi hT)}{k_{1}}\upsilon _{2}-\frac{2k_{1}^{\prime
}k_{2}+k_{1}k_{2}^{\prime }}{\lambda }\upsilon _{3}.
\end{equation*}%
Since $\xi $ is a unit vector field, we have $\eta (\upsilon _{2})^{2}+\eta
(\upsilon _{3})^{2}=1.$ The converse statement is trivial. Hence, we complete the
proof.
\end{proof}

If the osculating order $r\geq 4$, then we can state the following theorem:

\begin{theorem}
\label{propcase3proper}Let $\gamma :I\subseteq
\mathbb{R}
\rightarrow M$ be a non-geodesic Legendre curve of osculating order $r\geq 4$
in a non-Sasakian contact metric manifold. Then $\gamma $ is a curve with $%
\mathcal{C}$-proper mean curvature vector field if and only if it satisfies%
\begin{equation*}
k_{1}=\text{ constant,}
\end{equation*}%
\begin{equation*}
\lambda =\frac{k_{1}^{2}\left( k_{1}^{2}+k_{2}^{2}\right) }{g(T,\varphi hT)}
\end{equation*}%
\begin{equation*}
\xi =\frac{g(T,\varphi hT)}{k_{1}}\upsilon _{2}-\frac{k_{1}k_{2}^{\prime }}{%
\lambda }\upsilon _{3}-\frac{k_{1}k_{2}k_{3}}{\lambda }\upsilon _{4}
\end{equation*}
and
\begin{equation*}
\eta (\upsilon _{2})^{2}+\eta(\upsilon _{3})^{2}++\eta(\upsilon _{4})^{2}=1.
\end{equation*}
\end{theorem}

\begin{proof}
Since $\gamma $ has $\mathcal{C}$-proper mean curvature vector field, by the
use of (\ref{cproper}), we have
\begin{equation}
3k_{1}k_{1}^{\prime }T+\left( k_{1}^{3}+k_{1}k_{2}^{2}-k_{1}^{\prime \prime
}\right) \upsilon _{2}-(2k_{1}^{\prime }k_{2}+k_{1}k_{2}^{\prime })\upsilon
_{3}-k_{1}k_{2}k_{3}\upsilon _{4}=\lambda \xi .  \label{m1}
\end{equation}%
Taking the inner product of (\ref{m1}) with $T$, we have $k_{1}k_{1}^{\prime
}=0$. Since $\gamma $ is not a geodesic, we find $k_{1}^{\prime }=0$, which
gives us $k_{1}$ is a constant. Now taking the inner product of (\ref{m1})
with $\upsilon _{2}$ and using (\ref{etaE2}), we find
\begin{equation*}
\lambda =\frac{k_{1}^{2}\left( k_{1}^{2}+k_{2}^{2}\right) }{g(T,\varphi hT)}.
\end{equation*}%
Taking the inner product of (\ref{m1}) with $\upsilon _{3}$ and $\upsilon
_{4},$ we get
\begin{equation}
\eta (\upsilon _{3})=-\frac{k_{1}k_{2}^{\prime }}{\lambda }  \label{etaE33}
\end{equation}%
and
\begin{equation}
\eta (\upsilon _{4})=-\frac{k_{1}k_{2}k_{3}}{\lambda },  \label{etaE4}
\end{equation}%
respectively. Since $\xi \in span\left\{ \upsilon _{2},\upsilon
_{3},\upsilon _{4}\right\} ,$ using (\ref{etaE33}) and (\ref{etaE4}), we
obtain
\begin{equation*}
\xi =\frac{g(T,\varphi hT)}{k_{1}}\upsilon _{2}-\frac{k_{1}k_{2}^{\prime }}{%
\lambda }\upsilon _{3}-\frac{k_{1}k_{2}k_{3}}{\lambda }\upsilon _{4}.
\end{equation*}%
Since $\xi $ is a unit vector field, we have $\eta (\upsilon _{2})^{2}+\eta(\upsilon _{3})^{2}++\eta(\upsilon _{4})^{2}=1.$
The converse statement is trivial. Thus we get the result as required.
\end{proof}

In the normal bundle, we can give the following theorem:

\begin{theorem}
\label{propcase3properNORMAL}Let $\gamma :I\subseteq
\mathbb{R}
\rightarrow M$ be a non-geodesic Legendre curve of osculating order $r\geq 4$
in a non-Sasakian contact metric manifold. Then $\gamma $ is a curve with $%
\mathcal{C}$-proper mean curvature vector field in the normal bundle if and
only if
\begin{equation*}
\lambda =\frac{k_{1}^{2}k_{2}^{2}-k_{1}k_{1}^{\prime \prime }}{g(T,\varphi
hT)},
\end{equation*}%
\begin{equation*}
\xi =\frac{g(T,\varphi hT)}{k_{1}}\upsilon _{2}-\frac{2k_{1}^{\prime
}k_{2}+k_{1}k_{2}^{\prime }}{\lambda }\upsilon _{3}-\frac{k_{1}k_{2}k_{3}}{%
\lambda }\upsilon _{4}
\end{equation*}
and
\begin{equation*}
\eta (\upsilon _{2})^{2}+\eta(\upsilon _{3})^{2}++\eta(\upsilon _{4})^{2}=1.
\end{equation*}
\end{theorem}

\begin{proof}
The proof is similar to the proof of Theorem \ref{propcase3proper}.
\end{proof}

\section{Examples} \label{Sect-Examp}

Let us take $M=%
\mathbb{R}
^{3}$ and denote the standard coordinate functions with $(x,y,z)$. We define
the following vector fields on $%
\mathbb{R}
^{3}$:%
\begin{equation*}
e_{1}=\frac{\partial }{\partial x},\text{ }e_{2}=\frac{\partial }{\partial y}%
,\text{ }e_{3}=2y\frac{\partial }{\partial x}+\left( \frac{1}{4}%
e^{2x}-y^{2}\right) \frac{\partial }{\partial y}+\frac{\partial }{\partial z}%
.
\end{equation*}%
It is seen that $e_{1}$, $e_{2}$, $e_{3}$ are linearly independent at all
points of $M$. We define a Riemannian metric on $M$ such that $e_{1}$, $%
e_{2} $, $e_{3}$ are orthonormal. Then we have%
\begin{equation*}
\left[ e_{1},e_{2}\right] =0,\text{ }\left[ e_{1},e_{3}\right] =\frac{e^{2x}%
}{2}e_{2},\text{ }\left[ e_{2},e_{3}\right] =-2ye_{2}+2e_{1}.
\end{equation*}%
Let $\eta $ be defined by $\eta (W)=g\left( W,e_{1}\right) $ for all $W\in
\chi (M)$. Let $\varphi $ be the $(1,1)$-type tensor field defined by $%
\varphi e_{1}=0$, $\varphi e_{2}=e_{3}$, $\varphi e_{3}=-e_{2}$. Then $%
\left( M,\varphi ,e_{1},\eta ,g\right) $ is a contact metric manifold. Let
us set $\xi =e_{1},$ $X=e_{2}$ and $\varphi X=e_{3}$. Let $\nabla $ be the
Levi-Civita connection corresponding to $g$ which is calculated as%
\begin{equation}
\begin{array}{ccc}
\nabla _{X}\xi =\left( -\frac{e^{2x}}{4}-1\right) \varphi X, & \nabla
_{\varphi X}\xi =\left( 1-\frac{e^{2x}}{4}\right) X, & \nabla _{\xi }\xi =0,
\\
\nabla _{\xi }X=\left( -\frac{e^{2x}}{4}-1\right) \varphi X, & \nabla _{\xi
}\varphi X=\left( 1+\frac{e^{2x}}{4}\right) X, & \nabla _{X}X=2y\varphi X,
\\
\nabla _{X}\varphi X=-2yX+\left( \frac{e^{2x}}{4}+1\right) \xi , & \nabla
_{\varphi X}X=\left( \frac{e^{2x}}{4}-1\right) \xi , & \nabla _{\varphi
X}\varphi X=0.%
\end{array}
\label{Levi-Civita}
\end{equation}%
By the definition of $h$, it is easy to see that
\begin{equation*}
\begin{array}{ccc}
hX=\frac{e^{2x}}{4}X, &  & h\varphi X=-\frac{e^{2x}}{4}\varphi X.%
\end{array}%
\text{ }
\end{equation*}%
Hence, $M$ is a $\left( \kappa ,\mu ,\nu \right) $-contact metric manifold
with $\kappa =1-\frac{e^{4x}}{16}$, $\mu =2\left( 1+\frac{e^{2x}}{4}\right) $%
, $\nu =2$ \cite{MP}.
\begin{example}
Let $M$ be the $\left( \kappa ,\mu ,\nu \right) $-contact metric manifold
given above and let $\gamma :I\subseteq
\mathbb{R}
\rightarrow M$ be a curve parametrized by $\gamma (s)=(\ln 2,0,\frac{\sqrt{2}%
}{2}s)$, where $s$ is the arc-length parameter on an open interval $I$. The
unit tangent vector field $T$ along $\gamma $ is%
\begin{equation*}
T=-\frac{\sqrt{2}}{2}X+\frac{\sqrt{2}}{2}\varphi X.
\end{equation*}%
Since $\eta (T)=0$, the curve is Legendre. Using (\ref{Levi-Civita}), we find%
\begin{equation*}
\nabla _{T}T=-\xi ,
\end{equation*}%
which gives us $k_{1}=1$ and $\upsilon _{2}=-\xi $. Differentiating $%
\upsilon _{2}$ along the curve $\gamma $, we have%
\begin{eqnarray*}
\nabla _{T}\upsilon _{2} &=&-\sqrt{2}\varphi X \\
&=&-k_{1}T+k_{2}\upsilon _{3}.
\end{eqnarray*}%
Thus, we get%
\begin{equation*}
k_{2}=1,\text{ }\upsilon _{3}=-\frac{\sqrt{2}}{2}\left( X+\varphi X\right) .
\end{equation*}%
Finally we find%
\begin{equation*}
g\left( T,\varphi hT\right) =-1.
\end{equation*}%
From Theorem \ref{case2cproperNORMAL}, $\gamma $ has $\mathcal{C}$-proper
mean curvature vector field in the normal bundle with $\lambda =-1.$
\end{example}
\bigskip
Let $M=E(2)$ be the group of rigid motions of Euclidean $2$-space with left
invariant Riemannian metric $g$. Then $M$ admits its compatible
left-invariant contact Riemannian structure if and only if there exists an
orthonormal basis $\left\{ e_{1},e_{2},e_{3}\right\} $ of the Lie algebra
such that \cite{Perrone}:%
\[
\begin{array}{ccc}
\left[ e_{1},e_{2}\right] =2e_{3}, & \left[ e_{2},e_{3}\right] =c_{2}e_{1},
& \left[ e_{3},e_{1}\right] =0,%
\end{array}%
\]%
where we choose $c_{2}>0$. The Reeb vector field $\xi $ is obtained by left
translation of $e_{3}$. The contact distribution $D$ is spanned by $e_{1}$
and $e_{2}$. Then using Koszul's formula, we have the following relations:%
\begin{equation}
\begin{array}{cc}
\nabla _{e_{1}}e_{2}=\frac{1}{2}\left( -c_{2}+2\right) e_{3}, & \nabla
_{e_{1}}e_{3}=-\frac{1}{2}\left( -c_{2}+2\right) e_{2}, \\
\nabla _{e_{2}}e_{1}=-\frac{1}{2}\left( c_{2}+2\right) e_{3}, & \nabla
_{e_{2}}e_{3}=\frac{1}{2}\left( c_{2}+2\right) e_{1}, \\
\nabla _{e_{3}}e_{1}=\frac{1}{2}\left( c_{2}-2\right) e_{2}, & \nabla
_{e_{3}}e_{2}=-\frac{1}{2}\left( c_{2}-2\right) e_{1},%
\end{array}
\label{Levi-Civita2}
\end{equation}%
all others are zero (for more details see \cite{Perrone} and \cite{Inoguchi}%
). Let us denote by $X=e_{1}$, $\varphi X=e_{2}$, $\xi =e_{3}$. By the
definition of $h$, it is easy to see that
\begin{equation}
\begin{array}{ccc}
hX=-\frac{1}{2}c_{2}X, &  & h\varphi X=\frac{1}{2}c_{2}\varphi X.%
\end{array}%
\text{ }  \label{hX}
\end{equation}

Let $\gamma :I\rightarrow M=E(2)$ be a unit speed Legendre curve with Frenet
frame $\left\{ T=\upsilon _{1},\upsilon _{2},\upsilon _{3}\right\} $. Let us
write%
\[
T=T_{1}\xi +T_{2}X+T_{3}\varphi X.
\]%
Since $\gamma $ is Legendre, $T_{1}=0$. Using (\ref{Levi-Civita2}), we find%
\begin{eqnarray*}
\nabla _{T}T &=&-T_{2}T_{3}c_{2}\xi +T_{2}^{\prime }X+T_{3}^{\prime }\varphi
X \\
&=&k_{1}\upsilon _{2}.
\end{eqnarray*}%
If we choose $\upsilon _{2}=\xi ,$ then%
\[
k_{1}=-T_{2}T_{3}c_{2}\text{, }
\]%
and we can take $T_{2}=-\cos \theta =$constant, $T_{3}=-\sin \theta =$%
constant such that $\cos \theta \sin \theta <0$. So we have
\begin{equation}
T=-\cos \theta X-\sin \theta \varphi X.  \label{T}
\end{equation}%
Then using (\ref{nablaXksi}), (\ref{Levi-Civita2}), (\ref{hX}) and (\ref{T}%
), we can write
\begin{equation}
\nabla _{T}\upsilon _{2}=\nabla _{T}\xi =-\frac{1}{2}\sin \theta \left(
c_{2}+2\right) X+\frac{1}{2}\cos \theta \left( -c_{2}+2\right) \varphi X.
\label{E2}
\end{equation}%
Moreover $g(T,\varphi hT)=-\sin \theta \cos \theta c_{2}=k_{1}.$

\bigskip So, we can state the following example:

\begin{example}
Let\textbf{\ }$M=E(2)$ be the group of rigid motions of Euclidean $2$-space
with left invariant Riemannian metric $g$ and has a compatible
left-invariant contact Riemannian structure given above. Let $\gamma
:I\rightarrow M$ be a unit speed Legendre curve of osculating order $2$ and $%
\left\{ T=\upsilon _{1},\upsilon _{2}=\xi \right\} $ the Frenet frame of $%
\gamma .$ Then $\gamma $ is a Legendre circle with curvature $k_{1}=-\cos
\theta \sin \theta c_{2}$, where the tangent vector field of $\gamma $ is $%
T=-\cos \theta X-\sin \theta \varphi X$ and $\theta $ is a constant such
that $\sin \theta \cos \theta <0.$

\medskip Moreover, we have%
\[
g(T,\varphi hT)=-\sin \theta \cos \theta c_{2}=k_{1}.
\]%
From Theorem \ref{case1propTANGENT}, $\gamma $ has $\mathcal{C}$-proper mean
curvature vector field with $\lambda =-\sin ^{3}\theta \cos ^{3}\theta
c_{2}^{3}.$
\end{example}

Now let us assume that $\gamma :I\rightarrow M=E(2)$ is a unit speed
Legendre curve of osculating order $3$ with Frenet frame $\left\{ T=\upsilon
_{1},\upsilon _{2}=\xi ,\upsilon _{3}\right\} $. Similar to the above
example, if we choose $\upsilon _{2}=\xi ,$ we find $k_{1}=-T_{2}T_{3}c_{2}$
and we can take $T=\cos \theta X+\sin \theta \varphi X$, where $\theta $ is
a constant such that $\sin \theta \cos \theta <0.$ Define a croos product $%
\times $ by $e_{1}\times e_{2}=e_{3}$. So we have $\upsilon _{3}=T\times \xi
=\sin \theta e_{1}-\cos \theta e_{2}$. Then using (\ref{Levi-Civita2}), we
obtain
\[
\nabla _{T}\upsilon _{3}=-\frac{1}{2}\left( \cos ^{2}\theta (-c_{2}+2)+\sin
^{2}\theta (c_{2}+2)\right) e_{3},
\]%
which gives us $k_{2}=\frac{1}{2}\left( \cos ^{2}\theta (-c_{2}+2)+\sin
^{2}\theta (c_{2}+2)\right) =$constant.

Hence, we have the following example:

\begin{example}
Let\textbf{\ }$M=E(2)$ be the group of rigid motions of Euclidean $2$-space
with left invariant Riemannian metric $g$ and has a compatible
left-invariant contact Riemannian structure given above. Let $\gamma
:I\rightarrow M$ be a unit speed Legendre curve of osculating order $2$ and $%
\left\{ T=\upsilon _{1},\upsilon _{2}=\xi ,\upsilon _{3}\right\} $ the
Frenet frame of $\gamma .$ Then $\gamma $ is a Legendre helix with
curvatures $k_{1}=-\cos \theta \sin \theta c_{2}$ and  $k_{2}=\frac{1}{2}%
\left( \cos ^{2}\theta (-c_{2}+2)+\sin ^{2}\theta (c_{2}+2)\right) ,$ where
the tangent vector field of $\gamma $ is $T=\cos \theta X+\sin \theta
\varphi X$ and $\theta $ is a constant such that $\sin \theta \cos \theta <0.
$

\medskip Moreover, we have%
\[
g(T,\varphi hT)=-\sin \theta \cos \theta c_{2}=k_{1}.
\]
From Theorem \ref{case2propTANGENT}, $\gamma $ has $\mathcal{C}$-proper mean
curvature vector field in the normal bundle with $\lambda =k_{1}\left(
k_{1}^{2}+k_{2}^{2}\right) .$ Furthermore, from Theorem \ref%
{case2cproperNORMAL}, $\gamma $ has $\mathcal{C}$-proper mean curvature
vector field in the normal bundle with $\lambda =k_{1}k_{2}^{2}.$

\end{example}



\begin{thebibliography}{33}

\bibitem{Arroyo} J. Arroyo, M. Barros, O. J.  Garay,  A characterization of
helices and Cornu spirals in real space forms, Bull. Austral. Math. Soc. 56
(1997)  37--49.

\bibitem{KA} B. K\i l\i \c{c}, K. Arslan, Harmonic $1$-type submanifolds
of Euclidean spaces, Int. J. Math. Stat. 3 (2008)  47--53.

\bibitem{BB-92} C. Baikoussis, D. E. Blair, Integral surfaces of Sasakian
space forms, J. Geom. 43 (1992)  30--40.

\bibitem{BB} C. Baikoussis, D. E. Blair, On Legendre curves in contact $3$%
-manifolds, Geom. Dedicata 49 (1994)  135--142.

\bibitem{Blair} D. E. Blair, Riemannian geometry of contact and symplectic
manifolds, Birkhauser, Boston, 2002.

\bibitem{Chen-89} B.-Y. Chen, Null $2$-type surfaces in Euclidean space,
Algebra, analysis and geometry (Taipei, 1988), 1--18, World Sci. Publ.,
Teaneck, NJ, 1989.

\bibitem{Chen-95} B.-Y. Chen, Submanifolds in de Sitter space-time
satisfying $\Delta H=\lambda H$, Israel J. Math. 91 (1995)  373--391.

\bibitem{CIL} J. T. Cho, J. Inoguchi, J.-E. Lee, On slant curves in Sasakian
$3$-manifolds, Bull. Austral. Math. Soc. 74 (2006), no. 3, 359--367.

\bibitem{GO} \c{S}. G\"{u}ven\c{c}, C. \"{O}zg\"{u}r, On slant curves in
trans-Sasakian manifolds, Rev. Un. Mat. Argentina 55 (2014)  81--100.

\bibitem{Inoguchi-2004} J. Inoguchi, Submanifolds with harmonic mean
curvature vector field in contact $3$-manifolds, Colloq. Math. 100
(2004)  163--179.

\bibitem{Inoguchi} J. Inoguchi, Biminimal submanifolds in contact $3$%
-manifolds, Balkan J. Geom. Appl. 12 (2007) 56--67.

\bibitem{KH} H. Kocayi\u{g}it, H. H. Hac\i saliho\u{g}lu, $1$\textit{-}type
curves and biharmonic curves in Euclidean $3$-space, Int. Electron. J. Geom.
4 (2011) 97--101.

\bibitem{LSL} J.-E. Lee, Y. J. Suh, H. Lee, $C$\textit{-}parallel mean
curvature vector fields along slant curves in Sasakian $3$-manifolds,
Kyungpook Math. J. 52 (2012)  49--59.

\bibitem{MP} M. Markellos, V. J. Papantoniou, Biharmonic submanifolds in
non-Sasakian contact metric $3$-manifolds, Kodai Math. J. 34 (2011)
144--167.

\bibitem{Perrone} D. Perrone, Homogeneous contact Riemannian
three-manifolds, Illinois J. Math. 13 (1997)  243--256.

\bibitem{Yano} K. Yano, M. Kon, Structures on manifolds, Series in Pure
Mathematics, 3. World Scientific Publishing Co., Singapore, 1984.
\end{thebibliography}
\end{document}